\documentclass[12pt]{amsart}
\usepackage[top=3cm, bottom=3.5cm, left=3cm, right=3cm]{geometry}
\usepackage{amsmath,amsfonts,amssymb,amsthm}
\usepackage{mathtools}
\usepackage{tikz-cd}
\usepackage[utf8]{inputenc}
\usepackage{hyperref}
\usepackage{enumerate}
\usepackage[utf8]{inputenc} % set input encoding (not needed with XeLaTeX)
\usepackage{mathrsfs,amsmath}
\usepackage{amsfonts}
\usepackage{tikz-cd}
\usepackage{geometry} % to change the page dimensions
\usepackage{graphicx} % support the \includegraphics command and options

 \usepackage[parfill]{parskip} % Activate to begin paragraphs with an empty line rather than an indent

\usepackage{booktabs} % for much better looking tables
\usepackage{array} % for better arrays (eg matrices) in maths
\usepackage{paralist} % very flexible & customisable lists (eg. enumerate/itemize, etc.)
\usepackage{verbatim} % adds environment for commenting out blocks of text & for better verbatim
\usepackage{subfig} % make it possible to include more than one captioned figure/table in a single float
\usepackage{fancyhdr} % This should be set AFTER setting up the page geometry
\newtheorem{prop}{Proposition}

\newtheorem{thm}{Theorem}
\newtheorem{corollary}{Corollary}
\newtheorem{lemma}{Lemma}

\newtheorem{question}{Question}
\newtheorem{remark}{Remark}

\DeclareMathOperator\R{R}
\DeclareMathOperator\Ric{Ric}
\DeclareMathOperator\psh{psh}
\DeclareMathOperator\tr{tr}

\begin{document}
\bibliographystyle{amsplain}

\author{Xi Sisi Shen}
\address{Department of Mathematics\\
  Northwestern University\\
  2033 Sheridan Road, Evanston, IL 60208}
\email[X. S. Shen]{xss@math.northwestern.edu}

\title{Estimates for metrics of constant Chern scalar curvature}

\begin{abstract}
We prove a priori estimates for constant Chern scalar curvature metrics on a compact complex manifold conditional on an upper bound on the entropy, extending a recent result by Chen-Cheng in the K\"ahler setting.
\end{abstract}
\maketitle

\section{Introduction}
Calabi introduced extremal K\"ahler metrics \cite{calabi1} as critical points of the $L^2$ norm of the curvature tensor, now known as the Calabi functional, in his search for the ``best" canonical metric in a given K\"ahler class. K\"ahler-Einstein and constant scalar curvature metrics are examples of extremal metrics. Existence of K\"ahler-Einstein metrics was proved independently by Yau \cite{yau78} and Aubin \cite{aubin78} for manifolds of negative first Chern class and by Yau \cite{yau78} for those of zero first Chern class. For manifolds of positive first Chern class (Fano manifolds), the Yau-Tian-Donaldson conjecture asserts that K-stability is a necessary and sufficient condition for existence a of K\"ahler-Einstein metric. The sufficiency was established by Chen-Donaldson-Sun \cite{cds1,cds2,cds3}, building on the work of Tian-Yau \cite{tian-yau87}, Tian \cite{tian87} in the case of Fano surfaces. The reverse implication was shown by Tian \cite{tian97}, Donaldson \cite{donaldson}, Stoppa \cite{stoppa} and the most general form by Berman \cite{berman}. The literature in the field is vast and we refer the reader to the surveys \cite{darvas, donaldsonbook, pss, gabor2} for references and some recent developments.

The Yau-Tian-Donaldson conjecture for constant scalar curvature K\"ahler metrics, abbreviated cscK, remains open; while it is known that cscK implies K-stability \cite{stoppa, berman}, the converse is still not settled. A recent breakthrough by Chen-Cheng \cite{cc17} addressed the existence of a cscK metric within a given K\"ahler class using the continuity path of Chen \cite{chen15} (see also \cite{h15,z16}). Chen-Cheng established a priori estimates under the assumption of a uniform upper bound for entropy, given by $$\text{Ent}(\tilde{\omega},\omega) = \int_X \log\tfrac{\tilde{\omega}^n}{\omega^n}\tilde{\omega}^n.$$ We note that the entropy is automatically bounded below since the map $x \mapsto x \log x$ for $x>0$ has a lower bound. Using their estimates, Chen-Cheng prove in \cite{cc18} that the properness of K-energy in terms of $L^1$ geodesic distance implies the existence of a cscK metric. In addition, they show that for manifolds with discrete automorphism group, non-increasing K-energy and the existence of a destablized geodesic ray is equivalent to the non-existence of cscK.   Chen-Cheng's work has been extended by He \cite{he1, he2} to the cases of Sasaki metrics and extremal metrics.

This paper addresses the question of whether the above theory can be extended to the non-K\"ahler complex setting. Indeed there has been a surge of interest recently in extending the study of geometric PDEs to the non-K\"ahler setting \cite{acs2, hlt, ppz, popovici, streets, stw, tw15, yury, yang, yzz}.

Let $X$ be a compact complex manifold of complex dimension $n$ and define a Hermitian metric $g$ on $X$ to be a smooth tensor such that  $(g_{i\bar{j}})$ is a positive definite Hermitian matrix at each point of $X$. Associate to $g$ a real $(1,1)$-form $\omega$ given by $$\omega = \sqrt{-1}g_{i\bar{j}}dz^i\wedge d\overline{z^j}$$ which we will also refer to as a Hermitian metric. Define the Chern scalar curvature of $\omega$ by $$\R(\omega) = -g^{i\bar{j}}\partial_i\partial_{\bar{j}}\log\det g.$$  It is natural to ask:

\begin{question}  Let $(X, \omega)$ be a compact Hermitian manifold.
Under what conditions does there exist a constant Chern scalar curvature metric of the form $\tilde{\omega}=\omega+\sqrt{-1}\partial\bar{\partial}\varphi$ for a smooth function $\varphi$?
\end{question}

A different problem is to look for a Hermitian metric with constant Chern scalar curvature within a given Hermitian conformal class, and this was investigated by Angella-Calamai-Spotti \cite{acs}.

In this paper, we seek to make progress towards answering Question 1. We prove a generalization of the Chen-Cheng estimates in the non-K\"ahler setting, under an assumption of the $\partial\bar{\partial}$-closedness of the metric $\omega$ and its square.  Namely:

\begin{thm}
Let $(X^n, \omega)$ be a compact Hermitian manifold of dimension $n$ such that $\omega$ satisfies $\partial\bar{\partial}\omega^k = 0$ for $k=1,2$. If $\tilde{\omega}= \omega +\sqrt{-1}\partial\bar{\partial}\varphi$ is a constant Chern scalar curvature Hermitian metric on $X$ for smooth potential function $\varphi$ then for all k, there exists $C(k)$ depending only on $(X,\omega)$ and upper bound for $\emph{Ent}(\tilde{\omega},\omega)$ such that $||\varphi||_{C^k(X,\omega)}\le C(k).$
\end{thm}

In our proof, the assumption that the given Hermitian metric $\omega$ satisfies $\partial\bar{\partial}\omega^k=0$ for $k=1,2$ ensures that the average Chern scalar curvature $\underline{\R}$ for the metric remains unchanged up to addition of $\sqrt{-1}\partial\bar{\partial}\varphi$ and preserves some other useful integral properties. In the case of complex surfaces, this assumption is very natural since it coincides with the metric being Gauduchon and it is a well-known result by Gauduchon that every Hermitian metric is conformal to a Gauduchon metric \cite{gauduchon77}. We plan on using these estimates towards building an existence theory for constant Chern scalar curvature metrics in subsequent work. 

The constant Chern scalar curvature Hermitian metric $\tilde{\omega} = \omega +\sqrt{-1}\partial\bar{\partial}\varphi$ equation can be written as the following coupled equations:
 \begin{align}\begin{split}\label{ccsc_eqn}
& \ F = \log\tfrac{\tilde{\omega}^n}{\omega^n} \\
\tilde{\Delta}F &= -\underline{\R}+\tr_{\tilde{\omega}} \Ric(\omega)
\end{split} \end{align}
where $\tilde{\Delta}$ and $\tr_{\tilde{\omega}}$ denote the Chern Laplacian and trace with respect to $\tilde{\omega}$, respectively.

Our proof of Theorem 1 follows the basic outline of Chen-Cheng \cite{cc17}. However, difficulties arise from the non-K\"ahlerity of $\omega$. To prove the theorem, it is sufficient to prove that $\tilde{\omega}$ is quasi-isometric to $\omega$ since all higher derivatives of $\varphi$ can then be obtained by a straightforward bootstrapping method (see Proposition 1.2 of \cite{cc17}) where we use the result in \cite{tssy} for the $C^{2,\alpha}$ estimate since we are working in the non-K\"ahler setting. 

We cover several well-known identities for covariant derivatives, curvature and torsion and establish the notation and conventions used in this paper in Section 2. 

In Section 3, we secure $C^0$ bounds on $\varphi$ and $F$ in terms of $(X,\omega)$ and the entropy following the sequence of arguments from \cite{cc17}, but using instead a non-K\"ahler generalization of Yau's theorem \cite{cherrier, tw09}, a non-K\"ahler generalization of Tian's $\alpha$-invariant, and a uniform estimate in the non-K\"ahler setting by Dinew-Ko\l{}odziej \cite{dk09} and B\l{}ocki \cite{blocki2}. 

A bound on the gradient of $\varphi$ depending only on $(X,\omega)$ and the entropy is established in Section 4 by applying a maximum principle to a modified quantity from that of Chen-Cheng \cite{cc17} to account for the new torsion terms that arise.

In Section 5, we obtain an $L^p$ bound on $\tr_\omega \tilde{\omega}$ depending only on $p$, $(X,\omega)$ and the entropy using an inequality by Cherrier \cite{cherrier} from the study of the non-K\"ahler complex Monge-Amp\`ere equation (see also (9.5) of \cite{tw15}) and a modified quantity from that of Chen-Cheng \cite{cc17} to provide control over torsion terms. We use a step involving integration by parts and note that an additional term arises from the derivative landing on the volume form since the volume form is not assumed to be closed.

Finally, we bound $\tr_\omega \tilde{\omega}$ depending on $L^p$ bounds in Section 6 following the method of \cite{cc17}. In order to control several bad terms arising from torsion, we make a very specific choice of the quantity to which we apply the maximum principle. From this, we obtain the bounds needed for the Moser iteration (see Section 4 of \cite{cc17}) that lead us to the desired $L^\infty$ bound on $\tr_\omega \tilde{\omega}$, with the $L^p$ bound from Section 5 serving as the base case for the iteration. This bound immediately gives us the $L^\infty$ bound on $\tr_{\tilde{\omega}} \omega$ since we have bounds on $F=\log \tfrac{\tilde{\omega}^n}{\omega^n}$ from Section 3, proving the quasi-isometry of $\omega$ and $\tilde{\omega}$.

\section{Preliminaries}

In this section, for the convenience of the reader, we include several well-known identities that will be needed for computations in the subsequent sections  (see also Section 2 of \cite{tw15}). 

Let $X$ be a compact complex manifold of complex dimension $n$. In this paper, we will frequently compute in complex coordinates $z^1,\ldots, z^n$ and write tensors in terms of this coordinate system. Let $g=g_{i\bar{j}}$ be a Hermitian metric on $X$ with associated $(1,1)$-form $\omega=\sqrt{-1}g_{i\bar{j}}dz^i\wedge d\overline{z^j}$ where all repeated indices are to understood as being summed from $1$ to $n$. We will often also refer to $\omega$ as a Hermitian metric. 

Let $\nabla$ be the \textit{Chern connection} associated to $g$, defined for a $(1,0)$-form $a = a_kdz^k$ as
\begin{align}\begin{split}
&\nabla_i a_k= \partial_i a_k - \Gamma_{ik}^j a_j\ , \ \ \  \nabla_i \overline{a_k} = \partial_i \overline{a_k}\label{cov_to_partial}
\end{split}\end{align}
and for a vector field $X=X^k\partial_k$ as
\begin{align*}
\nabla_i X^k = \partial_i X^k +\Gamma^k_{ij}X^j\ , \ \ \ \nabla_i \overline{X^k} = \partial_i\overline{X^k}
\end{align*}
where $\Gamma^k_{ij} = g^{k\bar{p}}\partial_i g_{j\bar{p}}$ is the Christoffel symbol of $g$ and $g^{k\bar{p}}g_{i\bar{p}} = \delta_{ik}$. For a function $f$, $\nabla_i f = \partial_i f$.
The Chern connection is compatible with the metric $g$ in the sense that $\nabla_k g_{i\bar{j}}=0$ $\forall i,j,k$.

The metric $\omega$ defines a pointwise norm on any tensor. Given $a, X$ as above we have that $$|a|^2_\omega = g^{i\bar{j}}a_i\overline{a_j} \ , \ \ |X|^2_\omega = g_{i\bar{j}}X^i \overline{X^j}.$$ For a tensor $Y^{i\bar{k}}_m$, we have that $|Y|^2_\omega = g_{i\bar{j}}g_{\ell\bar{k}}g^{m\bar{n}}Y^{i\bar{k}}_m \overline{Y^{j\bar{\ell}}_n}$.

We define the \textit{trace} of a real $(1,1)$-form $\alpha=\alpha_{i\bar{j}}dz^i\wedge d\overline{z^j}$ with respect to $\omega$ by $$\tr_\omega \alpha = g^{i\bar{j}}\alpha_{i\bar{j}} = \tfrac{n\omega^{n-1}\wedge \alpha}{\omega^n}.$$

The \textit{curvature tensor} is defined as
\begin{align*}
R_{i\bar{j}k}^{\;\;\;\;\;p} = -\partial_{\bar{j}}\Gamma^p_{ik} \ , \ \ \ R_{i\bar{j}k\bar{\ell}} = g_{p\bar{\ell}}R_{i\bar{j}k}^{\;\;\;\;\;p}
\end{align*}
where we note that $\overline{R_{i\bar{j}k\bar{\ell}}} = R_{j\bar{i}\ell\bar{k}}$. 

The \textit{torsion} of $g$ is defined by $$T^k_{ij} = \Gamma^k_{ij} - \Gamma^k_{ji}.$$

We have the following formulae for commuting indices of the curvature tensor:
\begin{align}\begin{split}\label{curvature_comm_formula}
R_{i\bar{j}k}^{\;\;\;\;\;p} - R_{k\bar{j}i}^{\;\;\;\;\;p} = \partial_{\bar{j}}\Gamma^p_{ki} - \partial_{\bar{j}}\Gamma^p_{ik} = \partial_{\bar{j}}T^p_{ki}\\
R_{i\bar{j}\;\;\bar{p}}^{\;\;\;\bar{k}} - R_{i\bar{p}\;\;\bar{j}}^{\;\;\;\bar{k}} = \partial_{i}\overline{\Gamma^k_{pj}} - \partial_{i}\overline{\Gamma^k_{jp}} = \partial_{i}\overline{T^k_{pj}} .
\end{split}\end{align}

We write the \textit{Chern-Ricci curvature} of $\omega$ as 
\begin{align*}
R_{i\bar{j}} = g^{k\bar{\ell}}R_{i\bar{j}k\bar{\ell}} = -\partial_i \partial_{\bar{j}} \log\det g,
\end{align*}
its associated form as
\begin{align*}
\Ric(\omega)=\sqrt{-1}R_{i\bar{j}}dz^i\wedge d\overline{z^j}
\end{align*}
and its \textit{Chern scalar curvature} as
\begin{align*}
\R(\omega) = g^{i\bar{j}}R_{i\bar{j}} = \tr_\omega \Ric(\omega).
\end{align*}

Let $\tilde{\omega} = \omega+\sqrt{-1}\partial\bar{\partial}\varphi$ be another Hermitian metric on $X$. From this definition, it is clear that $$(\partial \omega)_{jk\bar{\ell}} = (\partial \tilde{\omega})_{jk\bar{\ell}}$$ where $(\partial\omega)_{jk\bar{\ell}} = \partial_j g_{k\bar{\ell}} - \partial_k g_{j\bar{\ell}}$. Denoting the torsion of $\tilde{\omega}$ by $\tilde{T}$, it follows that
\begin{align}\begin{split}\label{torsion}
T^p_{jk}g_{p\bar{\ell}} = (\partial\omega)_{jk\bar{\ell}} &= (\partial\tilde{\omega})_{jk\bar{\ell}} = \tilde{T}^q_{jk}\tilde{g}_{q\bar{\ell}}\\
\overline{T^p_{j\ell}}g_{k\bar{p}} = (\bar{\partial}\omega)_{\bar{j}k\bar{\ell}}&=(\bar{\partial}\tilde{\omega})_{\bar{j}k\bar{\ell}}= \overline{\tilde{T}^q_{j\ell}}\tilde{g}_{k\bar{q}}.
\end{split}
\end{align}
where $\tilde{g}_{i\bar{j}}$ is the metric in coordinates for $\tilde{\omega}$.

For simplicity, we will use the notation $\tilde{T}_{jk\bar{\ell}}=\tilde{T}^p_{jk}\tilde{g}_{p\bar{\ell}}$ and $T_{jk\bar{\ell}}=T^q_{jk}g_{q\bar{\ell}}$ and so the above equality can be rewritten as $\tilde{T}_{jk\bar{\ell}}=T_{jk\bar{\ell}}$.

We provide some commutation formulae which we will need for computations in the next few sections. For a $(1,0)$-form $a=a_kdz^k$ , we have
\begin{align}\begin{split}\label{a_comm_formula}
[\nabla_i, \nabla_{\bar{j}}]a_k &= -R_{i\bar{j}k\;}^
{\;\;\;\;\;\ell}a_{\ell}\\
[\nabla_i,\nabla_{\bar{j}}]\overline{a_{l}} &= R_{i\bar{j}\;\;\bar{\ell}}^{\;\;\;\bar{k}}\overline{a_{k}}\\
[\nabla_i,\nabla_j] \overline{a_k} &= -T^r_{ij}\nabla_r \overline{a_k}\\
[\nabla_{\bar{i}},\nabla_{\bar{j}}]a_k &= -\overline{T^r_{ij}}\nabla_{\bar{r}}a_k
\end{split}\end{align}
and for a scalar function $f$, we have 
\begin{align}\begin{split}\label{f_comm_formula}
[\nabla_i,\nabla_j] f &= -T_{ij}^r \nabla_r f \\
[\nabla_{\bar{i}},\nabla_{\bar{j}}] f &= -\overline{T_{ij}^r}\nabla_{\bar{r}} f  .
\end{split}\end{align}

The \textit{Chern Laplacian} with respect to $g$ of a function $f$ is defined as
\begin{align*}
\Delta f = \tr_\omega \sqrt{-1}\partial\bar{\partial}f = g^{i\bar{j}}\partial_i \partial_{\bar{j}} f = g^{i\bar{j}} \nabla_i \nabla_{\bar{j}} f  .
\end{align*}

For a complex manifold, if we assume that 
\begin{align}
\partial\bar{\partial}\omega^k = 0 \ \text{  for  } \ k=1,2, \label{dd_assumption}
\end{align}
then in fact it vanishes for all $k=1,\ldots, n-1$, following from a straightforward computation. Under this assumption, 
 \begin{align*}
\int_X (\omega+\sqrt{-1}\partial\bar{\partial}\psi)^n = \int_X \omega^n
\end{align*}
for any $\psi\in \psh(X, \omega)$ where
\begin{align*}
\psh(X,\omega) = \{\varphi\in C^\infty (X): \omega+\sqrt{-1}\partial\bar{\partial}\varphi >0 \}
\end{align*} and ensures the vanishing of the integrals of Chern Laplacians of functions:
\begin{align*}
\int_X \Delta f \omega^n = n\int_X \sqrt{-1}\partial\bar{\partial} f \wedge\omega^{n-1} = n \int_X f \sqrt{-1}\partial\bar{\partial}\omega^{n-1} = 0 .
\end{align*}
Our assumption from \eqref{dd_assumption} also gives us that the average Chern scalar curvature quantity $\underline{\R}$ is invariant under addition of $\sqrt{-1}\partial\bar{\partial}\varphi$ for any smooth function $\varphi$ since 
\begin{align*}
\underline{\R(\tilde{\omega})} &= \tfrac{\int_X \R(\tilde{\omega})\tilde{\omega}^n}{\int_X\tilde{\omega}^n} = \tfrac{\int_X n \Ric(\tilde{\omega})\wedge \tilde{\omega}^{n-1}}{\int_X \tilde{\omega}^n}\\
&= \tfrac{n\int_X (\Ric(\omega)-\sqrt{-1}\partial\bar{\partial} F)\wedge (\omega+ \sqrt{-1}\partial\bar{\partial}\varphi)^{n-1}}{\int_X \tilde{\omega}^n}\\
&=\tfrac{n\int_X \Ric(\omega)\wedge \omega^{n-1}}{\int_X \omega^n}=\underline{\R(\omega)}=\underline{\R}
\end{align*}
does not depend on $\varphi$, where we used the fact that $\tilde{\omega}^n = e^F\omega^n$ and that the Chern-Ricci form is closed.

These properties will be necessary for the proofs in the later sections. Note that throughout this paper, the constants may vary from line to line.

\section{$C^0$ bounds on $F$ and $\varphi$ in terms of the entropy}
In this section, we prove that an upper bound on the entropy implies $C^0$ bounds for $\varphi$ and $F$. We follow the sequence of arguments of Chen-Cheng \cite{cc17} employing, where necessary, the non-K\"ahler generalizations of the original theorems.
In particular, we show:
\begin{lemma}
Let $(\varphi, F)$ be a smooth solution to \eqref{ccsc_eqn}, then there exists a $C$ depending only on  $(X,\omega)$ and an upper bound on $\emph{Ent}(\tilde{\omega},\omega)$ such that $||F||_0+||\varphi||_0\le C$.\label{entropy_bound_suff}
\end{lemma}

In particular, the proof relies on a non-K\"ahler generalization of Yau's theorem by Cherrier \cite{cherrier} and Tosatti-Weinkove \cite{tw09}, a non-K\"ahler generalization of Tian's $\alpha$-invariant and a result by Dinew-Ko\l{}odziej \cite{dk09} and B\l{}ocki \cite{blocki2} (see also \cite{kolodziej, blocki} for the original K\"ahler results).

The non-K\"ahler generalization of Yau's theorem proved by Tosatti-Weinkove \cite{tw09} can be stated as follows:
\begin{corollary}\label{hermitian-yau}
For every smooth real-valued function $G$ on $X$ there exist a unique real number $b$ and a unique smooth real-valued function $\psi$ on $X$ solving
\begin{align*}
&(\omega+\sqrt{-1}\partial\bar{\partial}\psi)^n = e^{G+b}\omega^n, \\
\text{with}& \ \ \omega + \sqrt{-1}\partial\bar{\partial}\psi >0 ,  \ \ \ \sup_X \psi = 0 .
\end{align*}
In particular, when $\partial\bar{\partial}\omega^k = 0$,  for $k=1,2$, then the constant $b$ must equal
\begin{align*}
\log \tfrac{\int_X \omega^n}{\int_X e^G \omega^n} .
\end{align*}
\end{corollary}
The following lemma by H\"ormander (see Proposition 4.2.9 in \cite{hormander2}) will be needed in the proof of the non-K\"ahler generalization of Tian's $\alpha$-invariant:
\begin{lemma}\label{hormander}
There exists a constant $C$ such that for every $\psi\in C^\infty(X)$ satisfying $\sqrt{-1}\partial\bar{\partial}\psi\ge 0$ and $\psi(z)\le 0$ in $\{|z|<1\}\subset\mathbb{C}^n$ with $\psi(0)\ge -1$, we have
\begin{align*}
\int\limits_{\{|z|<1/2\}} e^{-\psi(z)}d\lambda(z)\le C .
\end{align*}
\end{lemma}
We are now ready to provide a proof of the generalized Tian's $\alpha$-invariant for Hermitian metrics for the convenience of the reader and which we believe is known to experts:
\begin{prop}
Given $(X,\omega)$ a Hermitian manifold, there exist constants $\alpha>0$ and $C>0$ depending only on $(X,\omega)$ such that 
 \begin{align*}
\int_X e^{-\alpha(\psi - \sup_X \psi)}\omega^n \le C
\end{align*}
for all $\psi \in \psh(X,\omega)$.\label{hermitian-tian}
\end{prop}
\begin{proof}
Following the argument by Tian \cite{tian87}, let us cover $X$ with $N$ geodesic balls $B_{16r}(x_i)$ with respect to $\omega$ such that $\cup_i B_r (x_i)$ covers $X$, with $N$ and $r$ uniform. Let us assume that each $B_{16r}(x_i)$ is contained in a holomorphic coordinate chart, $(U, \{z^j\})$, rescaled in $r$ and $\{z^j\}$ so that for all $w\in B_r(x_i)$, we have that
\begin{align*}
B_{2r}(w) \subset \{|z-w|\le 1/2\} \subset B_{4r}(w) \subset \{|z-w|\le 1\} \subset B_{8r}(w).
\end{align*}
By a result in \cite{tw09} (see also Proposition 2.1 in \cite{dk09}), having $\sup_X \psi = 0$ and $\psi\in \psh(X,\omega)$ implies that there is a uniform $L^1$ bound on $\psi$ in $B_{16r}(x_i)$. Hence, there exists a point $y_i\in B_r(x_i)$ such that
\begin{align*}
\psi(y_i)\ge -C 
\end{align*}
for a uniform $C$.
Then, we have that
\begin{align*}
B_r(x_i)\subset B_{2r}(y_i) \subset \{|z-y_i|\le 1/2\}\subset \{|z-y_i|\le 1\} \subset B_{8r}(y_i)\subset B_{16r}(x_i)
\end{align*}
and, in particular, on $\{|z-y_i|<1\}$ we have that 
\begin{align*}
\tfrac{\psi(y_i)}{C} \ge -1 \ \ \ \ \ \tfrac{\psi}{C} \le 0 .
\end{align*}
By the result by H\"ormander (Lemma \ref{hormander}), it follows that 
\begin{align*}
\int\limits_{|z-y_i|<1/2} e^{-\psi(z)/C}d\lambda(z)\le C .
\end{align*}
From this, we obtain that
\begin{align*}
\int\limits_{B_r(x_i)} e^{-\psi(z)/C}d\lambda(z) \le \int\limits_{\{|z-y_i|<1/2\}} e^{-\psi(z)/C}d\lambda(z) \le C .
\end{align*}
Since this holds on each of the $N$ balls with which we have covered $X$, we are done.
\end{proof}

Given Corollary \ref{hermitian-yau}, Proposition \ref{hermitian-tian} and a result by Dinew-Ko\l{}odziej \cite{dk09} and B\l{}ocki \cite{blocki2}, Lemma \ref{entropy_bound_suff} follows verbatim from \cite{cc17}. We provide here a proof for the convenience of the reader.
\begin{proof} (of Lemma \ref{entropy_bound_suff}) 
Firstly, we will normalize $\varphi$ so that $\sup_X \varphi =0$ and $\omega$ such that $\int_X \omega^n = 1$. Then, taking $G=F\log\sqrt{F^2+1}$, we have by Corollary \ref{hermitian-yau} and the assumption in \eqref{dd_assumption} that there exists a unique function $\psi$ solving
\begin{align*}
(\omega+\sqrt{-1}\partial\bar{\partial}\psi)^n &= e^{G+b}\omega^n = \tfrac{e^F\sqrt{F^2+1}\omega^n}{\int_X e^F\sqrt{F^2+1}\omega^n}
\end{align*}
with $\omega+\sqrt{-1}\partial\bar{\partial}\psi >0, \ \ \sup_X \psi = 0$. By Proposition \ref{hermitian-tian}, there exists $\alpha>0$ such that
\begin{align}
\int_X e^{-\alpha\varphi}\omega^n \le C , \ \ \ \ \int_X e^{-\alpha\psi}\omega^n \le C . \label{a-invariant}
\end{align}
Let $\varepsilon,\delta,\theta \in (0,1)$ be constants to be determined. Let $p\in X$ and choose a coordinate ball $B(p)$. Let $\eta$ be a smooth cut-off function on $X$ such that $1-\theta\le\eta\le 1$ with $$\eta(p)=1, \ \eta|_{\partial B}=1-\theta, \ |\partial\eta|_\omega^2=O(\theta^2), \ |\nabla^2\eta|_{\omega}=O(\theta).$$ Let $Q:=e^{\delta(F+\varepsilon\psi-\lambda\varphi)}$ and $A:=\delta(F+\varepsilon\psi-\lambda\varphi)$. Assume that $Q$ attains a maximum at $p\in X$. In order to apply the Alexandrov-Bakelman-Pucci (ABP) maximum principle (see Lemma 9.3 in \cite{gt}), we need to compute
\begin{align*}
e^{-A}\tilde{\Delta}(Q\eta)&=(\tilde{\Delta}A + |\partial A|^2_{\tilde{\omega}})\eta + \tilde{\Delta}\eta + 2\text{Re}(\tilde{g}^{i\bar{j}}A_i\eta_{\bar{j}})
\end{align*}
\begin{align*}
\tilde{\Delta}A&= \delta\big(-\underline{\R}+\tr_{\tilde{\omega}} \Ric(\omega) +\varepsilon \tilde{g}^{i\bar{j}}(g_\psi)_{i\bar{j}}-\varepsilon \tilde{g}^{i\bar{j}}g_{i\bar{j}} -\lambda n + \lambda \tilde{g}^{i\bar{j}}g_{i\bar{j}}\big)\\
&\ge \delta\big(-(\underline{\R}+\lambda n) + (\lambda - C-\varepsilon) \tr_{\tilde{\omega}} \omega + \varepsilon n(\sqrt{F^2+1}I_F^{-1})^{1/n}\big)
\end{align*}
where we used the fact that $$\tilde{g}^{i\bar{j}}(g_\psi)_{i\bar{j}}\ge n\Big(\tfrac{\omega_\psi^n}{\tilde{\omega}^n}\Big)^{\tfrac{1}{n}}=n(\sqrt{F^2+1}I_F^{-1})^{\tfrac{1}{n}}$$ where $I_F = \int_X e^F \sqrt{F^2+1}\omega^n$. We have the following bounds:
\begin{align*}
\tilde{\Delta} \eta &\ge -(\tr_{\tilde{\omega}} \omega)O(\theta)\\
2\text{Re}(\tilde{g}^{i\bar{j}}A_i\eta_{\bar{j}}) &\ge -\eta|\partial A|^2_{\tilde{\omega}} -\tfrac{(\tr_{\tilde{\omega}} \omega) O(\theta^2)}{\eta} .
\end{align*}
Combining these inequalities together, and choosing $\lambda$ sufficiently large, $\delta$ such that $2n\delta\lambda=\alpha$ and $\theta$ small compared to $\delta$, we have
\begin{align*}
e^{-A}\tilde{\Delta} (Q\eta)&\ge \delta\eta(-\underline{\R}-\lambda n +n\varepsilon(\sqrt{F^2+1}I_F^{-1})^{1/n})  +\delta\eta(\tr_{\tilde{\omega}}\omega \big(\lambda-C-\varepsilon)\big)\\
& \ \ \ \ -\tr_{\tilde{\omega}}\omega\Big(O(\theta)+\tfrac{O(\theta^2)}{\eta}\Big)\\
&\ge \delta\eta(-\underline{\R}-\lambda n +\varepsilon n(\sqrt{F^2+1})^{1/n}I_F^{-1/n}) . 
\end{align*}

Applying ABP to $Q\eta=e^{\delta(F+\varepsilon\psi-\lambda\varphi)}\eta$, we have
\begin{align}\begin{split}
\sup_B & Q\eta \le \sup_{\partial B} Q\eta\\
&  + C_n\Big(\int_B \delta Q^{2n}e^{2F}((-\underline{\R}-\lambda n+\varepsilon n(\sqrt{F^2+1}I_F^{-1})^{1/n})^{-}) ^{2n} \omega^n\Big)^{1/2n} . \label{int}
\end{split}
\end{align}
The integral vanishes except when $-\underline{\R}-\lambda n+\varepsilon n(\sqrt{F^2+1}I_F^{-1}){1/n}<0$. By the positivity of $(\sqrt{F^2+1})^{1/n}$ and $I_F^{-1/n}$, we find that the integral on the right-hand side of \eqref{int} is bounded above by
\begin{align*}
\int\limits_{B\cap\{F\le C\}} \delta e^{2n\delta(F+\varepsilon\psi-\lambda\varphi)}e^{2F}(|\underline{\R}|+\lambda n) \omega^n &\le C\int\limits_{B\cap\{F\le C\}} e^{2n\delta(\varepsilon\psi-\lambda\varphi)} \omega^n\\
& \le C\int\limits_X e^{-2n\delta\lambda\varphi}\omega^n = C\int\limits_X e^{-\alpha\varphi}\omega^n \le C
\end{align*}
since $\psi\le 0$ and by applying \eqref{a-invariant}, where $C$ depends on $\varepsilon$ and $I_F$. This gives us that 
\begin{align}
Q(p) = \sup_X Q \le (1-\theta)\sup_X Q + C\Rightarrow F+\varepsilon\psi-\lambda\varphi\le C . \label{F_party_bound}
\end{align}
Now, in order to arrive at an upper bound on $F$, it suffices to prove $C^0$ bounds on $\psi$ and $\varphi$. A bound on $\varphi$ can be accomplished by showing that $e^F = \tfrac{\tilde{\omega}^n}{\omega^n} \in L^q(X)$ for $q>1$ and using a result of Dinew-Ko\l{}odziej \cite{dk09} and B\l{}ocki \cite{blocki2}.
By \eqref{a-invariant}, \eqref{F_party_bound} and the fact that $\varphi\le 0$, we have that
\begin{align*}
\int_X e^{\alpha F/\varepsilon}\omega^n \le \int_X e^{\alpha(C-\varepsilon\psi+\lambda\varphi)/\varepsilon}\omega^n\le \int_X e^{-\alpha\psi}\omega^n \le C .
\end{align*}
Choosing $\varepsilon$ such that $\varepsilon = \tfrac{\alpha}{q}$ for $q>1$, we arrive at an $L^q$ bound for $e^F$ which by the previously stated result gives us $||\varphi||_0\le C$. We also have that $$\tfrac{\omega_\psi^n}{\omega^n} = \tfrac{e^F\sqrt{F^2+1}}{\int_X e^F\sqrt{F^2+1}\omega^n} \le Ce^{qF}$$ for some $C>0$ and $q>1$ and so by the same argument, we also obtain bounds on $||\psi||_0$. Since $I_F$ can be bounded from above in terms of $\text{Ent}(\tilde{\omega},\omega)$, the dependence of the constant on $I_F$ passes over to $\text{Ent}(\tilde{\omega},\omega)$. Thus, we have shown an upper bound on $F$, as well as a $C^0$ bound on $\varphi$, as desired.

It remains to show a lower bound on $F$. For $K>0$ to be determined, we can compute
\begin{align*}
\tilde{\Delta}(F+K\varphi) &= -\underline{\R} + \tilde{g}^{i\bar{j}} R_{i\bar{j}}+ K\tilde{g}^{i\bar{j}}\varphi_{i\bar{j}} \le -\underline{\R} + Kn -(K-C)\tilde{g}^{i\bar{j}} g_{i\bar{j}}.
\end{align*}
Choosing $K>C$ and using the arithmetic-geometric mean inequality $$\tr_{\tilde{\omega}} \omega \ge n\big(\tfrac{\omega^n}{\tilde{\omega}^n}\big)^{\tfrac{1}{n}}=ne^{-\tfrac{F}{n}}\ ,$$ we find that at a minimum $p_0$ of $F+K\varphi$, we have
\begin{align*}
0\le -\underline{\R}+Kn - n(K-C)e^{-\tfrac{F}{n}}
\end{align*}
giving us the desired lower bound for $F$ in terms of $||\varphi||_0$ which can be bounded in terms of $(X,\omega)$ and $\text{Ent}(\tilde{\omega},\omega)$.
\end{proof}
\begin{remark}
In the paper by Chen-Cheng \cite{cc17}, they also bound the entropy in terms of $||\varphi||_0$ using the fact that a cscK metric is a minimizer of K-energy. In the Hermitian case, it is not known whether there exists a notion of K-energy and so it is unclear whether such an implication should hold.
\end{remark}
\section{Gradient bound on the potential}
In this section, we prove a bound on $|\partial\varphi|^2_\omega$ by applying a maximum principle argument to a modified quantity from that of Chen-Cheng \cite{cc17}. This gradient term appears in the computation for proving bounds on the $L^p$ norms of $\tr_\omega \tilde{\omega}$. We use the fact that we have secured $C^0$ bounds on $F$ and $\varphi$ depending only on $(X,\omega)$ and the entropy, as shown in the last section.
\begin{lemma}\label{gradient_varphi_bound}
Let $(\varphi, F)$ be a smooth solution to \eqref{ccsc_eqn}. Then there exists a constant $C$ depending only on $(X,\omega)$ and $\emph{Ent}(\tilde{\omega},\omega)$ such that
 \begin{align}\begin{split}
|\partial \varphi|^2_\omega\le C .
\end{split} \end{align}
\end{lemma}
\begin{proof}
Consider the quantity $Q:=e^{-(F+\lambda\varphi)+\tfrac{1}{2}\varphi^2}(|\partial \varphi|^2_\omega +1)$ and let $A:=-(F+\lambda\varphi)+\tfrac{1}{2}\varphi^2$.
We will compute $\tilde{\Delta} Q$ for $\lambda>0$ to be determined. 

Firstly, we have
 \begin{align}\begin{split}\label{grad_phi}
e^{-A}\tilde{\Delta} Q &= (\tilde{\Delta} A+|\partial A|^2_{\tilde{\omega}})(|\partial \varphi|^2_\omega + 1) 
+\tilde{\Delta} (|\partial \varphi|^2_\omega) + 2\text{Re}\big(\tilde{g}^{i\bar{j}}A_i(|\partial\varphi|^2_\omega)_{\bar{j}}\big) .
\end{split} \end{align}

Firstly, we have that
 \begin{align*}\begin{split}
\tilde{\Delta}A &= -\tilde{\Delta} F-\lambda\tilde{\Delta} \varphi + \tfrac{1}{2}\tilde{\Delta} \varphi^2\\
& = \underline{\R}-\tilde{g}^{i\bar{j}} R_{i\bar{j}} -(\lambda-\varphi)n + (\lambda-\varphi) \tilde{g}^{i\bar{j}}g_{i\bar{j}} + |\partial \varphi|^2_{\tilde{\omega}} .
\end{split} \end{align*}

Let $\nabla$ be the covariant derivative with respect to $g$. The second term in \eqref{grad_phi} can be computed as
\begin{align*}
\tilde{\Delta}(|\partial\varphi|^2_\omega)&=\tilde{g}^{i\bar{j}}\nabla_i \nabla_{\bar{j}} (g^{k\bar{\ell}}\varphi_k \varphi_{\bar{\ell}})\\
&=\tilde{g}^{i\bar{j}}g^{k\bar{\ell}}(\nabla_k\nabla_i\nabla_{\bar{j}}\varphi \varphi_{\bar{\ell}} +T_{ki}^r \nabla_r\nabla_{\bar{j}}\varphi \varphi_{\bar{\ell}}+ \nabla_i\varphi_k\nabla_{\bar{j}}\varphi_{\bar{\ell}}\\
& \ \ \ \ +\nabla_{\bar{j}}\varphi_k\nabla_i\varphi_{\bar{\ell}} + \varphi_k \nabla_i\nabla_{\bar{\ell}}\nabla_{\bar{j}}\varphi   + \varphi_k\partial_i(\overline{T_{lj}^r})\varphi_{\bar{r}}+\varphi_k\overline{T^r_{\ell j}}\varphi_{i\bar{r}})\\
&=g^{k\bar{\ell}}F_k\varphi_{\bar{\ell}} +2\text{Re}(\tilde{g}^{i\bar{j}}g^{k\bar{\ell}}T^r_{ki}\varphi_{r\bar{j}}\varphi_{\bar{\ell}}) + \tilde{g}^{i\bar{j}}g^{k\bar{\ell}}\varphi_{ki}\varphi_{\bar{\ell}\bar{j}} + \tilde{g}^{i\bar{j}}g^{k\bar{\ell}}\varphi_{k\bar{j}}\varphi_{\bar{\ell}i}\\
& \ \ \ \ +g^{k\bar{\ell}}\tilde{g}^{i\bar{j}}\varphi_k\nabla_{\bar{\ell}}\nabla_i\nabla_{\bar{j}}\varphi + g^{k\bar{\ell}}\tilde{g}^{i\bar{j}}\varphi_k R_{i\bar{\ell}\;\;\bar{j}}^{\;\;\;\bar{r}}\varphi_{\bar{r}}
+g^{k\bar{\ell}}\tilde{g}^{i\bar{j}}\varphi_k\partial_i(\overline{T^r_{\ell j}})\varphi_{\bar{r}}\\
&=2\text{Re}(g^{k\bar{\ell}}F_k\varphi_{\bar{\ell}}) +2\text{Re}(\tilde{g}^{i\bar{j}}g^{k\bar{\ell}}T^r_{ki}\varphi_{r\bar{j}}\varphi_{\bar{\ell}}) + \tilde{g}^{i\bar{j}}g^{k\bar{\ell}}\varphi_{ki}\varphi_{\bar{\ell}\bar{j}} + \tilde{g}^{i\bar{j}}g^{k\bar{\ell}}\varphi_{k\bar{j}}\varphi_{\bar{\ell}i}\\
& \ \ \ \ + g^{k\bar{\ell}}\tilde{g}^{i\bar{j}}\varphi_k R_{i\bar{\ell}\;\;\bar{j}}^{\;\;\;\bar{r}}\varphi_{\bar{r}}  +g^{k\bar{\ell}}\tilde{g}^{i\bar{j}}\varphi_k\partial_i(\overline{T^r_{\ell j}})\varphi_{\bar{r}},
\end{align*}
where we used the commutation formula from \eqref{a_comm_formula} and the fact that
\begin{align*}
\tilde{g}^{i\bar{j}}\nabla_k \nabla_i\nabla_{\bar{j}}\varphi = F_k.
\end{align*}
From there we commute the indices of the curvature tensor as in \eqref{curvature_comm_formula} and use the fact that 
\begin{align*}
F_k = -A_k -\lambda\varphi_k + \varphi\varphi_k
\end{align*}
to obtain
\begin{align*}
\tilde{\Delta}(|\partial\varphi|^2_\omega)&=-2\text{Re}(g^{k\bar{\ell}}A_k\varphi_{\bar{\ell}}) +2\text{Re}(\tilde{g}^{i\bar{j}}g^{k\bar{\ell}}T^r_{ki}\varphi_{r\bar{j}}\varphi_{\bar{\ell}}) + \tilde{g}^{i\bar{j}}g^{k\bar{\ell}}\varphi_{ki}\varphi_{\bar{\ell}\bar{j}} + \tilde{g}^{i\bar{j}}g^{k\bar{\ell}}\varphi_{k\bar{j}}\varphi_{\bar{\ell}i}\\
& \ \ \ \ + g^{k\bar{\ell}}\tilde{g}^{i\bar{j}}\varphi_kR_{i\bar{j}\;\;\bar{\ell}}^{\;\;\;\bar{r}}\varphi_{\bar{r}} -2(\lambda-\varphi)|\partial\varphi|^2_\omega .
\end{align*}
Substituting back into \eqref{grad_phi}, we arrive at the following equality:
\begin{align*}
e^{-A}\tilde{\Delta} Q &=(|\partial A|^2_{\tilde{\omega}}+ \big(\underline{\R}-(\lambda-\varphi)n  +\tilde{g}^{i\bar{j}}((\lambda -\varphi)g_{i\bar{j}}-R_{i\bar{j}})+|\partial \varphi|^2_{\tilde{\omega}}\big)(|\partial\varphi|^2_\omega+1)\\
& \ \ \ \  -2\text{Re}(g^{k\bar{\ell}}A_k\varphi_{\bar{\ell}}) +2\text{Re}(\tilde{g}^{i\bar{j}}g^{k\bar{\ell}}T^r_{ki}\varphi_{r\bar{j}}\varphi_{\bar{\ell}}) + \tilde{g}^{i\bar{j}}g^{k\bar{\ell}}\varphi_{ki}\varphi_{\bar{\ell}\bar{j}} + \tilde{g}^{i\bar{j}}g^{k\bar{\ell}}\varphi_{k\bar{j}}\varphi_{\bar{\ell}i} \\
& \ \ \ \  + g^{k\bar{\ell}}\tilde{g}^{i\bar{j}}\varphi_kR_{i\bar{j}\;\;\bar{\ell}}^{\;\;\;\bar{r}} \varphi_{\bar{r}}  -2(\lambda-\varphi)|\partial \varphi|^2_{\tilde{\omega}}+2\text{Re}(\tilde{g}^{i\bar{j}}A_ig^{k\bar{\ell}}(\varphi_k\varphi_{\bar{\ell}\bar{j}}+\varphi_{k\bar{j}}\varphi_{\bar{\ell}})).
\end{align*}
Now, we use the completed square
\begin{align*}
0&\le \tilde{g}^{i\bar{j}}g^{k\bar{\ell}}(\varphi_{ki}+A_i\varphi_k)(\varphi_{\bar{\ell}\bar{j}}+A_{\bar{j}}\varphi_{\bar{\ell}})\\
& = \tilde{g}^{i\bar{j}}g^{k\bar{\ell}}\varphi_{ki}\varphi_{\bar{\ell}\bar{j}} +|\partial A|^2_{\tilde{\omega}}|\partial\varphi|^2_\omega+2\text{Re}(\tilde{g}^{i\bar{j}}g^{k\bar{\ell}}\varphi_{\bar{\ell}\bar{j}}A_i\varphi_{k}),
\end{align*}
the simplification
\begin{align*}
g^{k\bar{\ell}}A_k \varphi_{\bar{\ell}}- g^{k\bar{\ell}}\tilde{g}^{i\bar{j}}A_i\varphi_{k\bar{j}}\varphi_{\bar{\ell}} = g^{k\bar{\ell}}\varphi_{\bar{\ell}}(A_k - \tilde{g}^{i\bar{j}}A_i(\tilde{g}_{k\bar{j}}-g_{k\bar{j}})) = \tilde{g}^{i\bar{j}}A_i\varphi_{\bar{j}}
\end{align*}
and rewrite the torsion term as
\begin{align*}
2\text{Re}(\tilde{g}^{i\bar{j}}g^{k\bar{\ell}}T^r_{ki}\varphi_{r\bar{j}}\varphi_{\bar{\ell}}) &=2\text{Re}(\tilde{g}^{i\bar{j}}g^{k\bar{\ell}}T^r_{ki}(\tilde{g}_{r\bar{j}}-g_{r\bar{j}})\varphi_{\bar{\ell}})\\
& =2\text{Re}(g^{k\bar{\ell}}T^i_{ki}\varphi_{\bar{\ell}})-2\text{Re}(\tilde{g}^{i\bar{j}}g^{k\bar{\ell}}g_{r\bar{j}}T^r_{ki}\varphi_{\bar{\ell}}),
\end{align*}
to obtain
\begin{align*}
e^{-A}\tilde{\Delta} Q&\ge |\partial A|^2_{\tilde{\omega}}+ (\underline{\R}-(\lambda-\varphi) n +\tilde{g}^{i\bar{j}}((\lambda -\varphi)g_{i\bar{j}}-R_{i\bar{j}})+|\partial \varphi|^2_{\tilde{\omega}})(|\partial\varphi|^2_\omega+1)\\
& \ \ \ \  +2\text{Re}(g^{k\bar{\ell}}T^i_{ki}\varphi_{\bar{\ell}})-2\text{Re}(\tilde{g}^{i\bar{j}}g^{k\bar{\ell}}g_{r\bar{j}}T^r_{ki}\varphi_{\bar{\ell}}) + \tilde{g}^{i\bar{j}}g^{k\bar{\ell}}\varphi_{k\bar{j}}\varphi_{\bar{\ell}i} + g^{k\bar{\ell}}\tilde{g}^{i\bar{j}}\varphi_kR_{i\bar{j}\;\;\bar{\ell}}^{\;\;\;\bar{r}} \varphi_{\bar{r}}\\
& \ \ \ \  -2(\lambda-\varphi)|\partial\varphi|^2_\omega  - 2\text{Re}(\tilde{g}^{i\bar{j}}A_i\varphi_{\bar{j}}).
\end{align*}

Applying a few instances of Young's inequality and choosing $\lambda$ sufficiently large, we have
\begin{align}
e^{-A}\tilde{\Delta} Q&\ge -C(|\partial\varphi|^2_\omega+1) +C\tr_{\tilde{\omega}}\omega(|\partial\varphi|^2_\omega+1) +|\partial \varphi|^2_{\tilde{\omega}}|\partial\varphi|^2_\omega  . \label{q_bound}
\end{align}
Noting an elementary consequence of the fact that $e^F =\tfrac{\tilde{\omega}^n}{\omega^n}$ (see page 12 in \cite{cc17}), we have the inequality
\begin{align*}
|\partial \varphi|^2_{\tilde{\omega}}|\partial\varphi|^2_\omega+|\partial\varphi|^2_\omega\tr_{\tilde{\omega}}\omega \ge \tfrac{1}{n-1}(|\partial\varphi|_{\omega}^2)^{1+\tfrac{1}{n}}e^{-\tfrac{F}{n}}.
\end{align*}
Applying this inequality to the last term in \eqref{q_bound}, we see that at a maximum of $Q$, we have the bound
\begin{align*}
0\ge (|\partial\varphi|^2_\omega)^{1+\tfrac{1}{n}}e^{-\tfrac{F}{n}}-C(|\partial\varphi|^2_\omega+1)  .
\end{align*}
Since we have bounds on $F$ depending on $(X,\omega)$ and the entropy, we arrive at the desired upper bound on $|\partial\varphi|^2_\omega$.
\end{proof}

\section{$L^p$ bound on the trace}
We are now ready to compute $L^p$ bounds on $\tr_\omega \tilde{\omega}$. Our approach reflects that of Chen-Cheng \cite{cc17} using a modification of the quantity to which we apply the maximum principle to account for new torsion terms. The result of this section will be crucial for obtaining the $L^\infty$ bound on the trace in the next section. We prove the following:
\begin{thm}
Let $(\varphi, F)$ be a smooth solution to \eqref{ccsc_eqn}. For any $p>0$, there exists a constant $C(p)$ depending only on $p$, $(X,\omega)$ and $\emph{Ent}(\tilde{\omega},\omega)$ such that 
 \begin{align*}
\int_X (\tr_\omega \tilde{\omega})^p \omega^n\le C(p).
\end{align*}
\end{thm}
\begin{proof}
Define $Q:= e^{-\alpha(F+\lambda\varphi)} (\tr_\omega \tilde{\omega}+1)$ and let $A:= -\alpha(F+\lambda\varphi)$ where $\alpha, \lambda>0$ are constants to be determined. We first compute
\begin{align*}
e^{-A}\tilde{\Delta} (e^A(\tr_\omega \tilde{\omega}+1))&=(\tilde{\Delta}A+|\partial A|^2_{\tilde{\omega}})(\tr_\omega \tilde{\omega} +1) + \tilde{\Delta} \tr_\omega \tilde{\omega} + 2\text{Re}(\tilde{g}^{i\bar{j}}A_i\partial_{\bar{j}}\tr_\omega \tilde{\omega}) .
\end{align*}
Using an inequality due to Cherrier \cite{cherrier} (see also (9.5) of \cite{tw15}) and the fact that $g$ is a fixed metric whose torsion terms and their derivatives are bounded by uniform constants, we have the following:
\begin{align*}\begin{split}
\tilde{\Delta} \log \tr_\omega \tilde{\omega} \ge \tfrac{1}{\tr_\omega \tilde{\omega}}\Big(2Re (\tilde{g}^{k\bar{q}}T^i_{ik}\tfrac{\partial_{\bar{q}}\tr_\omega \tilde{\omega}}{\tr_\omega \tilde{\omega}}) +\Delta F - C\tr_\omega \tilde{\omega} \tr_{\tilde{\omega}} \omega  \Big)\\
\Rightarrow \tilde{\Delta} \tr_\omega \tilde{\omega} \ge 2\text{Re} (\tilde{g}^{k\bar{q}}T^i_{ik}\tfrac{\partial_{\bar{q}}\tr_\omega \tilde{\omega}}{\tr_\omega \tilde{\omega}}) +\Delta F - C\tr_\omega \tilde{\omega} \tr_{\tilde{\omega}} \omega +\tfrac{|\partial\tr_\omega \tilde{\omega}|^2_{\tilde{\omega}}}{\tr_\omega \tilde{\omega}}
\end{split} \end{align*} 
where we used the fact that we have uniform lower bounds on $\tr_\omega \tilde{\omega}$ and $\tr_{\tilde{\omega}} \omega$ by the geometric-arithmetic mean inequality.
We will use the following completed square:
 \begin{align*}\begin{split}
0& \le \tfrac{1}{\tr_\omega \tilde{\omega}} \tilde{g}^{i\bar{j}}(A_i \tr_\omega\tilde{\omega}+ T^k_{ki}+\partial_i\tr_\omega\tilde{\omega})(A_{\bar{j}}\tr_\omega\tilde{\omega}+\overline{T^\ell_{\ell j}}+\partial_{\bar{j}}\tr_\omega\tilde{\omega})\\
& = |\partial A|^2_{\tilde{\omega}} \tr_\omega \tilde{\omega}+ \tfrac{\tilde{g}^{i\bar{j}}T^k_{ki}\overline{T^\ell_{\ell j}}}{\tr_\omega \tilde{\omega}}  + \tfrac{|\partial\tr_\omega \tilde{\omega}|^2_{\tilde{\omega}}}{\tr_\omega \tilde{\omega}} + 2 \text{Re}(\tilde{g}^{i\bar{j}}A_i \overline{T^\ell_{\ell j}})\\
& \ \ \ \  +2\text{Re}(\tilde{g}^{i\bar{j}}A_i\partial_{\bar{j}}\tr_\omega \tilde{\omega}) + \tfrac{2}{\tr_\omega \tilde{\omega}} \text{Re}(\tilde{g}^{i\bar{j}}T^k_{ki} \partial_{\bar{j}}\tr_\omega \tilde{\omega}) .
\end{split} \end{align*}
Putting this together, we have
\begin{align}\begin{split}\label{laplacian_sec3}
e^{-A}\tilde{\Delta} Q &\ge \tilde{\Delta}A(\tr_\omega \tilde{\omega}+1)+ |\partial A|^2_{\tilde{\omega}}(\tr_\omega \tilde{\omega}+1) + \tfrac{2}{\tr_\omega \tilde{\omega}}\text{Re}(\tilde{g}^{i\bar{j}}T^k_{ki}\partial_{\bar{j}}\tr_\omega\tilde{\omega})\\
& \ \ \ \  +\Delta F\ - C\tr_\omega \tilde{\omega} \tr_{\tilde{\omega}} \omega + \tfrac{|\partial \tr_\omega \tilde{\omega}|_{\tilde{\omega}}^2}{\tr_\omega \tilde{\omega}} + 2\text{Re}(\tilde{g}^{i\bar{j}}A_i\partial_{\bar{j}}\tr_\omega \tilde{\omega})\\
& \ge \alpha(\underline{\R} - \tr_{\tilde{\omega}}\Ric(\omega) - \lambda \tilde{\Delta} \varphi)(\tr_\omega \tilde{\omega}+1)+|\partial A|^2_{\tilde{\omega}} + \Delta F  \\
&\ \ \ \  - C\tr_\omega \tilde{\omega} \tr_{\tilde{\omega}} \omega  -\tfrac{\tilde{g}^{i\bar{j}}T^k_{ki}\overline{T^\ell_{\ell j}}}{\tr_\omega \tilde{\omega}}-2\text{Re}(\tilde{g}^{i\bar{j}}A_i\overline{T^{\ell}_{\ell\bar{j}}})\\
& \ge \alpha(\underline{\R} - \lambda n +(\tfrac{\lambda}{2}-\tfrac{C}{\alpha})\tr_{\tilde{\omega}} \omega)(\tr_\omega \tilde{\omega} +1)  +\Delta F   \\
\end{split} \end{align}
where we used in the last line the following instance of Young's inequality:
\begin{align*}
2 \text{Re}(\tilde{g}^{i\bar{j}}A_i\overline{T^{\ell}_{\ell\bar{j}}}) \ge -\tilde{g}^{i\bar{j}}T^k_{ki}\overline{T^{\ell}_{\ell j}}- |\partial A|^2_{\tilde{\omega}} \ge -C\tr_{\tilde{\omega}} \omega - |\partial A|^2_{\tilde{\omega}}
\end{align*}
and chose $\lambda$ sufficiently large compared to $\Ric(\omega)$.

Using the fact that 
 \begin{align*}\begin{split}
\tfrac{1}{2p+1}\tilde{\Delta} (Q^{2p+1}) =2pQ^{2p-1}|\partial Q|^2_{\tilde{\omega}} +Q^{2p}\tilde{\Delta} Q \ge 2pQ^{2p-2}|\partial Q|^2_\omega e^A +Q^{2p}\tilde{\Delta} Q
\end{split} \end{align*}
and integrating with respect to $\tilde{\omega}^n= e^F \omega^n$, we have at
 \begin{align*}\begin{split}
& \int_X  2pe^{A+F}Q^{2p-2}|\partial Q|^2_\omega \omega^n+ \int_X e^{A+F}(\tfrac{\lambda\alpha}{2} - C)\tr_{\tilde{\omega}}\omega(\tr_\omega \tilde{\omega} +1) Q^{2p}\omega^n \\
&\ \ \ \ \ \  + \int_X e^{A+F}Q^{2p}\Delta F \omega^n \le \int_X\alpha e^{A+F}(\lambda n-\underline{\R})(\tr_\omega \tilde{\omega}+1) Q^{2p}\omega^n .
\end{split} \end{align*}
Integrating by parts the integral involving $\Delta F$, where we note that an additional term arises from the derivative landing on the volume form, and using Young's inequality, we see that
 \begin{align*}\begin{split}
\int_X &e^{A+F} Q^{2p}\Delta F \omega^n = \int_X e^{(1-\alpha)F-\alpha\lambda\varphi}Q^{2p}\sqrt{-1}\partial\bar{\partial} F \wedge \omega^{n-1} \\
&\ge \int_X e^{(1-\alpha)F-\alpha\lambda\varphi}Q^{2p}\sqrt{-1} ((\alpha-1)\partial F +\alpha\lambda \partial\varphi)\wedge\bar{\partial} F  \wedge\omega^{n-1}\\
& \ \ \ \  - \int_X e^{(1-\alpha)F-\alpha\lambda\varphi}2pQ^{2p-1}\sqrt{-1}\partial Q \wedge \bar{\partial}F\wedge \omega^{n-1}  -C\int_X e^{A+F}Q^{2p}|\partial F|_{\omega} \omega^n\\
&\ge \int_X (\tfrac{\alpha-1}{2}-p-\tfrac{1}{2})Q^{2p}e^{A+F}|\partial F|^2_\omega \omega^n - \int_X \Big(\tfrac{C\alpha^2\lambda^2}{2(\alpha-1)}+C\Big)e^{A+F}Q^{2p}\omega^n\\
& \ \ \ \ \ \ \ \ \  - \int_X pe^{A+F}Q^{2p-2}|\partial Q|^2_\omega \omega^n
\end{split} \end{align*} 
where we used Lemma \ref{gradient_varphi_bound} to bound $|\partial\varphi|^2_\omega$ in the last inequality along with the fact that we have a lower bound on $Q$. 
Combining everything together and bounding $e^{A+F}$, we arrive at
 \begin{align*}\begin{split}
\int_X &pQ^{2p-2}|\partial Q|^2_\omega\omega^n + \int_X (\tfrac{\alpha}{2}- p -1)Q^{2p}|\partial F|^2_\omega \omega^n\\
& \ \  + \int_X (\tfrac{\lambda\alpha}{2} - C)\tr_\omega \tilde{\omega}(\tr_{\tilde{\omega}}\omega+1) Q^{2p}\omega^n\\
& \ \ \ \ \ \  \le C\int_X\alpha (\lambda n-\underline{\R})(\tr_\omega \tilde{\omega}+1) Q^{2p}\omega^n + C\int_X (1+\tfrac{\alpha^2\lambda^2}{2(\alpha-1)})Q^{2p}\omega^n .
\end{split} \end{align*}
Choosing $\lambda\ge 2C+2$ and requiring $\alpha\ge 2(p+2)$ and $p\ge 0$, we have that
 \begin{align*}\begin{split}
\int_X (\tr_\omega \tilde{\omega})^{2p+1+\tfrac{1}{n-1}}\omega^n &\le C\int_X\tr_\omega \tilde{\omega}(\tr_{\tilde{\omega}} \omega +1) Q^{2p}\omega^n\\
&  \le C\int_X (\tr_\omega \tilde{\omega}+1) Q^{2p}\omega^n + C\int_XQ^{2p}\omega^n\\
&\le C\int_X \tr_\omega \tilde{\omega} Q^{2p}\omega^n\le C\int_X (\tr_\omega \tilde{\omega})^{2p+1}\omega^n
\end{split} \end{align*}
where the third inequality holds since we have a lower bound for $\tr_\omega \tilde{\omega}$ and $C$ depends on $p, (X,\omega), \text{Ent}(\tilde{\omega},\omega)$. For the case $p=0$, we see that
 \begin{align*}\begin{split}
\int_X (\tr_\omega \tilde{\omega})^{1+\tfrac{1}{n-1}} \omega^n&\le C\int_X  \tr_\omega \tilde{\omega} \ \omega^n \le  C vol(X) .
\end{split} \end{align*}
Thus, by iterating, we can bound the $L^p$ norm of $\tr_\omega \tilde{\omega}$ by a constant depending on $p$, $(X,\omega)$ and $\text{Ent}(\tilde{\omega},\omega)$.
\end{proof}
\section{$L^\infty$ bound on the trace}
In this section we will obtain a uniform $L^\infty$ bound on $\tr_\omega \tilde{\omega}$. We will accomplish this by computing the $L^\infty$ norm of the sum of $\tr_\omega \tilde{\omega}$ and $|\partial F|^2_{\tilde{\omega}}$  as this will help cancel out some bad terms, following the strategy of Chen-Cheng \cite{cc17}. The key ingredient is a calculation using covariant derivatives with respect to $\tilde{\omega}$ for a specific quantity to which we apply the maximum principle. The quantity is chosen in such a way as to preserve a positive amount of certain desirable terms which will serve to control the bad terms that arise from torsion and derivatives of torsion.

In particular, we prove 
\begin{thm}
Let $(\varphi, F)$ be a smooth solution to \eqref{ccsc_eqn}. Then there exists a constant $C$ depending only on $(X,\omega)$ and $\emph{Ent}(\tilde{\omega},\omega)$, such that
\begin{align*}
\max_X (\tr_\omega \tilde{\omega}) + \max_X |\partial F|^2_{\tilde{\omega}} \le C .
\end{align*}
\end{thm}
\begin{proof}
Let $\tilde{\nabla}$, $\tilde{R}$ and $\tilde{T}$ denote, respectively, the covariant derivative, curvature tensor and torsion with respect to $\tilde{g}$. Commuting derivatives as in \eqref{a_comm_formula} and \eqref{f_comm_formula}, we have the following:
\begin{align}\begin{split}\label{laplacian_of_grad_norm}
\tilde{\Delta}(|\partial F|^2_{\tilde{\omega}})&=\tilde{g}^{i\bar{j}}\tilde{g}^{p\bar{q}}\big((\tilde{\nabla}_p \tilde{\nabla}_i \tilde{\nabla}_{\bar{j}}F + \tilde{T}^r_{pi}\tilde{\nabla}_r F_{\bar{j}})F_{\bar{q}} +\overline{\tilde{T}^r_{qj}}\tilde{\nabla}_i F_{\bar{r}}F_p + \tilde{\nabla}_i\overline{\tilde{T}^r_{qj}}F_{\bar{r}}F_p\\
& \ \ \ \ + (\tilde{\nabla}_{\bar{q}}\tilde{\nabla}_i \tilde{\nabla}_{\bar{j}}F + \tilde{R}_{i\bar{q}\; \bar{j}}^{\; \;\; \bar{\ell}}F_{\bar{\ell}})F_p  \big)  +  |\tilde{\nabla}\tilde{\nabla} F|^2_{\tilde{\omega}}+|\tilde{\nabla} \bar{\tilde{\nabla}} F|^2_{\tilde{\omega}}\\
&=\tilde{g}^{p\bar{q}}\big(\tilde{\Delta} F)_p F_{\bar{q}} + 2\text{Re}(\tilde{g}^{i\bar{j}}\tilde{g}^{p\bar{q}}\tilde{T}^r_{pi}\tilde{\nabla}_r F_{\bar{j}}F_{\bar{q}}) + \tilde{g}^{i\bar{j}}\tilde{g}^{p\bar{q}}\tilde{\nabla}_i\overline{\tilde{T}^r_{qj}}F_{\bar{r}}F_p\\
& \ \ \ \ + \tilde{g}^{p\bar{q}}(\tilde{\Delta} F)_{\bar{q}}F_p   + \tilde{g}^{i\bar{j}}\tilde{g}^{p\bar{q}}\tilde{R}_{i\bar{q}\; \bar{j}}^{\;\;\;\bar{\ell}} F_{\bar{\ell}}F_p+ |\tilde{\nabla}\tilde{\nabla} F|^2_{\tilde{\omega}}+|\tilde{\nabla} \bar{\tilde{\nabla}} F|^2_{\tilde{\omega}}\\
&=\tilde{g}^{p\bar{q}}\big(\tilde{\Delta} F)_p F_{\bar{q}} + 2\text{Re}(\tilde{g}^{i\bar{j}}\tilde{g}^{p\bar{q}}\tilde{g}^{r\bar{k}}\tilde{T}_{pi\bar{k}}\tilde{\nabla}_r F_{\bar{j}}F_{\bar{q}})\\
 & \ \ \ \ +\tilde{g}^{i\bar{j}}\tilde{g}^{p\bar{q}}\tilde{g}^{t\bar{r}}\tilde{\nabla}_i(\overline{\tilde{T}_{qj\bar{t}}})F_{\bar{r}}F_p + \tilde{g}^{p\bar{q}}(\tilde{\Delta} F)_{\bar{q}}F_p+ \tilde{g}^{i\bar{j}}\tilde{g}^{p\bar{q}} \tilde{g}^{k\bar{\ell}}\tilde{R}_{i\bar{q}k\bar{j}} F_{\bar{\ell}}F_p\\
 & \ \ \ \ + |\tilde{\nabla}\tilde{\nabla} F|^2_{\tilde{\omega}}+|\tilde{\nabla} \bar{\tilde{\nabla}} F|^2_{\tilde{\omega}}\\
&=2\text{Re}(\tilde{g}^{p\bar{q}}\big(\tilde{\Delta} F)_p F_{\bar{q}}) + 2\text{Re}(\tilde{g}^{i\bar{j}}\tilde{g}^{p\bar{q}}\tilde{g}^{r\bar{k}}\tilde{T}_{pi\bar{k}}\tilde{\nabla}_r F_{\bar{j}}F_{\bar{q}})\\
& \ \ \ \ +\tilde{g}^{i\bar{j}}\tilde{g}^{p\bar{q}}\tilde{g}^{t\bar{r}}\tilde{\nabla}_i(\overline{\tilde{T}_{qj\bar{t}}})F_{\bar{r}}F_p  + \tilde{g}^{p\bar{q}} \tilde{g}^{k\bar{\ell}}\tilde{R}_{k\bar{q}}F_{\bar{\ell}}F_p\\
& \ \ \ \  -\tilde{g}^{p\bar{q}}\tilde{g}^{k\bar{\ell}}\tilde{g}^{r\bar{s}}\tilde{\nabla}_{\bar{q}}(\tilde{T}_{rk\bar{s}})F_{\bar{\ell}}F_p+|\tilde{\nabla}\tilde{\nabla} F|^2_{\tilde{\omega}}+|\tilde{\nabla} \bar{\tilde{\nabla}} F|^2_{\tilde{\omega}} .
\end{split}\end{align}
For a general real-valued function $A(F)$, 
\begin{align*}
e^{-A(F)}\tilde{\Delta}(e^{A(F)} |\partial F|^2_{\tilde{\omega}})&= \tilde{\Delta}(|\partial F|^2_{\tilde{\omega}}) +2A'\text{Re}(\tilde{g}^{i\bar{j}}\tilde{g}^{k\bar{\ell}}(F_iF_k F_{\bar{\ell}\bar{j}} +F_iF_{\bar{\ell}}F_{k\bar{j}}))\\
& \ \ \ \ + (A'^2+A'')|\partial F|^4_{\tilde{\omega}} + A'\tilde{\Delta} F |\partial F|^2_{\tilde{\omega}} ,
\end{align*}
where we use the simplified notation $F_{\bar{\ell}\bar{j}}$ to denote $\tilde{\nabla}_{\bar{j}}\tilde{\nabla}_{\bar{\ell}}F$.
Substituting \eqref{laplacian_of_grad_norm} for the first term in the above equation and noting the following completed square:
\begin{align*}
A'^2|\partial F|^4_{\tilde{\omega}} + 2A'\text{Re}(\tilde{g}^{i\bar{j}}\tilde{g}^{k\bar{\ell}}F_iF_kF_{\bar{\ell}\bar{j}})+|\tilde{\nabla}\tilde{\nabla} F|^2_{\tilde{\omega}}\ge 0 ,
\end{align*}
we have that
\begin{align*}
e^{-A(F)}\tilde{\Delta}(e^{A(F)} |\partial F|^2_{\tilde{\omega}})&\ge 2\text{Re}(\tilde{g}^{p\bar{q}}(\tilde{\Delta} F)_p F_{\bar{q}}) +2\text{Re}(\tilde{g}^{i\bar{j}}\tilde{g}^{p\bar{q}}\tilde{g}^{r\bar{k}}\tilde{T}_{pi\bar{k}}\tilde{\nabla}_r F_{\bar{j}}F_{\bar{q}})\\
&\ \ \ \ +\tilde{g}^{i\bar{j}}\tilde{g}^{p\bar{q}}\tilde{g}^{t\bar{r}}\tilde{\nabla}_i(\overline{\tilde{T}_{qj\bar{t}}})F_{\bar{r}}F_p + \tilde{g}^{p\bar{q}} \tilde{g}^{k\bar{\ell}}\tilde{R}_{k\bar{q}}F_{\bar{\ell}}F_p\\
& \ \ \ \ -\tilde{g}^{p\bar{q}}\tilde{g}^{k\bar{\ell}}\tilde{g}^{r\bar{s}}\tilde{\nabla}_{\bar{q}}(\tilde{T}_{rk\bar{s}})F_{\bar{\ell}}F_p + |\tilde{\nabla} \bar{\tilde{\nabla}} F|^2_{\tilde{\omega}} + 2A'\tilde{g}^{i\bar{j}}\tilde{g}^{k\bar{\ell}}F_iF_{\bar{\ell}}F_{k\bar{j}}\\
& \ \ \ \ +A''|\partial F|^4_{\tilde{\omega}} +A'\tilde{\Delta} F|\partial F|^2_{\tilde{\omega}}.
\end{align*}
Switching the Ricci curvature of $\tilde{\omega}$ to that of $\omega$ using the relation
\begin{align*}
\tilde{R}_{k\bar{q}} = R_{k\bar{q}} - F_{k\bar{q}},
\end{align*} we arrive at 
\begin{align}\begin{split}\label{af_eqn}
e&^{-A(F)}\tilde{\Delta}(e^{A(F)} |\partial F|^2_{\tilde{\omega}})\\
&\ge 2\text{Re}(\tilde{g}^{p\bar{q}}(\tilde{\Delta} F)_p F_{\bar{q}}) +2\text{Re}(\tilde{g}^{i\bar{j}}\tilde{g}^{p\bar{q}}\tilde{g}^{r\bar{k}}\tilde{T}_{pi\bar{k}}\tilde{\nabla}_r F_{\bar{j}}F_{\bar{q}})+\tilde{g}^{i\bar{j}}\tilde{g}^{p\bar{q}}\tilde{g}^{t\bar{r}}\tilde{\nabla}_i(\overline{\tilde{T}_{qj\bar{t}}})F_{\bar{r}}F_p \\
& \ \ \ \  + \tilde{g}^{p\bar{q}} \tilde{g}^{k\bar{\ell}}R_{k\bar{q}}F_{\bar{\ell}}F_p -\tilde{g}^{p\bar{q}}\tilde{g}^{k\bar{\ell}}\tilde{g}^{r\bar{s}}\tilde{\nabla}_{\bar{q}}(\tilde{T}_{rk\bar{s}})F_{\bar{\ell}}F_p + |\tilde{\nabla} \bar{\tilde{\nabla}} F|^2_{\tilde{\omega}}\\
& \ \ \ \ + (2A'-1)\tilde{g}^{i\bar{j}}\tilde{g}^{k\bar{\ell}}F_iF_{\bar{\ell}}F_{k\bar{j}} +A''|\partial F|^4_{\tilde{\omega}} +A'\tilde{\Delta} F|\partial F|^2_{\tilde{\omega}}\\
&\ge 2\text{Re}(\tilde{g}^{p\bar{q}}(\tilde{\Delta} F)_p F_{\bar{q}}) +2\text{Re}(\tilde{g}^{i\bar{j}}\tilde{g}^{p\bar{q}}\tilde{g}^{r\bar{k}}\tilde{T}_{pi\bar{k}}\tilde{\nabla}_r F_{\bar{j}}F_{\bar{q}})+\tilde{g}^{i\bar{j}}\tilde{g}^{p\bar{q}}\tilde{g}^{t\bar{r}}\tilde{\nabla}_i(\overline{\tilde{T}_{qj\bar{t}}})F_{\bar{r}}F_p\\
& \ \ \ \ + \tilde{g}^{p\bar{q}} \tilde{g}^{k\bar{\ell}}R_{k\bar{q}}F_{\bar{\ell}}F_p -\tilde{g}^{p\bar{q}}\tilde{g}^{k\bar{\ell}}\tilde{g}^{r\bar{s}}\tilde{\nabla}_{\bar{q}}(\tilde{T}_{rk\bar{s}})F_{\bar{\ell}}F_p + (1-(A'-\tfrac{1}{2}))|\tilde{\nabla} \bar{\tilde{\nabla}} F|^2_{\tilde{\omega}}\\
& \ \ \ \ +(A'' - (A'-\tfrac{1}{2}))|\partial F|^4_{\tilde{\omega}} +A'\tilde{\Delta} F|\partial F|^2_{\tilde{\omega}}  
\end{split}\end{align}
where we used the following Cauchy-Schwarz inequality:
\begin{align*}
(2A'-1)\tilde{g}^{i\bar{j}}\tilde{g}^{k\bar{\ell}}F_iF_{\bar{\ell}}F_{k\bar{j}} &\ge -(A'-\tfrac{1}{2})|\partial F|^4_{\tilde{\omega}} - (A'-\tfrac{1}{2})|\tilde{\nabla}\bar{\tilde{\nabla}} F|^2_{\tilde{\omega}}
\end{align*}
for $A'>\tfrac{1}{2}$.

In order to control the bad torsion terms (the second, third and fifth terms in the last line of \eqref{af_eqn}), we will need to specifically choose our function $A(F)$ to ensure that $1-(A'-\tfrac{1}{2})>0$ and $A''-(A'-\tfrac{1}{2})>0$. We can accomplish this by choosing $$A(F)=\kappa e^{F} +F(\tfrac{1}{2}-\varepsilon),$$ so that $A'(F)=\kappa e^{F}+\tfrac{1}{2}-\varepsilon$ and $A''(F)= \kappa e^F$. We then can choose $\varepsilon, \kappa>0$ such that
$$0 \le A'' - \varepsilon = A' -\tfrac{1}{2} \le \tfrac{1}{2}$$\[ \Leftrightarrow \begin{cases} 
	\kappa e^{\min_X F}-\varepsilon \ge 0\\
	\kappa e^{\max_X F}-\varepsilon\le \tfrac{1}{2}\ .
   \end{cases}
\]

We can first choose $\kappa$ small enough such that $\kappa e^{\max_X F}\le \tfrac{1}{2}$. Then choose $\varepsilon$ small enough such that $\kappa e^{\min F}\ge \varepsilon$. This ensures that $A' \in(\tfrac{1}{2},1)$.

It follows that
\begin{align*}
e&^{-A(F)}\tilde{\Delta}(e^{A(F)}|\partial F|_{\tilde{\omega}}^2)\\
& \ge 2\text{Re}\big(\tilde{g}^{p\bar{q}}\big(\tilde{\Delta} F)_p F_{\bar{q}}\big)+ 2\text{Re}(\tilde{g}^{i\bar{j}}\tilde{g}^{p\bar{q}}\tilde{g}^{r\bar{k}}\tilde{T}_{pi\bar{k}}\tilde{\nabla}_r F_{\bar{j}}F_{\bar{q}}) +\tilde{g}^{i\bar{j}}\tilde{g}^{p\bar{q}}\tilde{g}^{t\bar{r}}\tilde{\nabla}_i(\overline{\tilde{T}_{qj\bar{t}}})F_{\bar{r}}F_p \\
& \ \ \ \   + \tilde{g}^{p\bar{q}} \tilde{g}^{k\bar{\ell}}R_{k\bar{q}}F_p F_{\bar{\ell}} -\tilde{g}^{p\bar{q}}\tilde{g}^{k\bar{\ell}}\tilde{g}^{r\bar{s}}\tilde{\nabla}_{\bar{q}}(\tilde{T}_{rk\bar{s}})F_{\bar{\ell}}F_p  + \tfrac{1}{2}|\tilde{\nabla} \bar{\tilde{\nabla}} F|^2_{\tilde{\omega}}+ \varepsilon |\partial F|_{\tilde{\omega}}^4 + A'\tilde{\Delta} F |\partial F|_{\tilde{\omega}}^2 \\
& \ge 2\text{Re}\big(\tilde{g}^{p\bar{q}}\big(\tilde{\Delta} F)_p F_{\bar{q}}\big)+ 2\text{Re}(\tilde{g}^{i\bar{j}}\tilde{g}^{p\bar{q}}\tilde{g}^{r\bar{k}}T_{pi\bar{k}}\tilde{\nabla}_r F_{\bar{j}}F_{\bar{q}}) +\tilde{g}^{i\bar{j}}\tilde{g}^{p\bar{q}}\tilde{g}^{t\bar{r}}\partial_i \overline{T_{qj\bar{t}}}F_{\bar{r}}F_p  \\
& \ \ \ \ - \tilde{g}^{i\bar{j}}\tilde{g}^{p\bar{q}}\tilde{g}^{t\bar{r}}\tilde{g}^{s\bar{k}}\partial_i \tilde{g}_{t\bar{k}}\overline{T_{qj\bar{s}}}F_{\bar{r}}F_p+ \tilde{g}^{p\bar{q}} \tilde{g}^{k\bar{\ell}}R_{k\bar{q}}F_p F_{\bar{\ell}} -\tilde{g}^{p\bar{q}}\tilde{g}^{k\bar{\ell}}\tilde{g}^{r\bar{s}}\partial_{\bar{q}}(T_{rk\bar{s}})F_{\bar{\ell}}F_p \\
& \ \ \ \  +\tilde{g}^{p\bar{q}}\tilde{g}^{k\bar{\ell}}\tilde{g}^{r\bar{s}}\tilde{g}^{i\bar{j}}\partial_{\bar{q}} \tilde{g}_{i\bar{s}}T_{rk\bar{j}}F_{\bar{\ell}}F_p  + \tfrac{1}{2}|\tilde{\nabla}\bar{\tilde{\nabla}}F|^2_{\tilde{\omega}} + \varepsilon |\partial F|^4_{\tilde{\omega}} - |\tilde{\Delta} F ||\partial F|^2_{\tilde{\omega}} 
\end{align*}
where we converted the covariant derivatives to partial derivatives as in \eqref{cov_to_partial} and passed the torsion terms of $\tilde{g}$ to those of $g$ as in \eqref{torsion}.
 
We can rewrite the first term appearing on the right hand side of the above inequality by using the following:
\begin{align*}
(\tilde{\Delta} F)_p &= \partial_p (\tilde{g}^{i\bar{j}}R_{i\bar{j}}) = -\tilde{g}^{a\bar{j}}\partial_p \tilde{g}_{a\bar{b}}\tilde{g}^{i\bar{b}}R_{i\bar{j}} + \tilde{g}^{i\bar{j}}\partial_p R_{i\bar{j}} \ .
\end{align*}

Putting this all together, applying Young's inequality and choosing $B$ to be at least $3(n-1)$, where the factor of $n-1$ comes from the fact that $\tr_\omega \tilde{\omega} \le C(\tr_{\tilde{\omega}}\omega)^{n-1}$, we have
\begin{align*}
e&^{-A(F)}\tilde{\Delta}(e^{A(F)}|\partial F|^2_{\tilde{\omega}})\\
&\ge -2\text{Re}(\tilde{g}^{p\bar{q}}\tilde{g}^{a\bar{j}}\partial_p\tilde{g}_{a\bar{b}}\tilde{g}^{i\bar{b}}R_{i\bar{j}}F_{\bar{q}})+2\text{Re}(\tilde{g}^{p\bar{q}}\tilde{g}^{i\bar{j}}\partial_p R_{i\bar{j}}F_{\bar{q}}) + 2\text{Re}(\tilde{g}^{i\bar{j}}\tilde{g}^{p\bar{q}}\tilde{g}^{r\bar{k}}T_{pi\bar{k}}\tilde{\nabla}_r F_{\bar{j}}F_{\bar{q}})\\
& \ \ \ \  +\tilde{g}^{i\bar{j}}\tilde{g}^{p\bar{q}}\tilde{g}^{t\bar{r}}\partial_i \overline{T_{qj\bar{t}}}F_{\bar{r}}F_p  - \tilde{g}^{i\bar{j}}\tilde{g}^{p\bar{q}}\tilde{g}^{t\bar{r}}\tilde{g}^{s\bar{k}}\partial_i \tilde{g}_{t\bar{k}}\overline{T_{qj\bar{s}}}F_{\bar{r}}F_p+ \tilde{g}^{p\bar{q}} \tilde{g}^{k\bar{\ell}}R_{k\bar{q}}F_p F_{\bar{\ell}} \\
& \ \ \ \  -\tilde{g}^{p\bar{q}}\tilde{g}^{k\bar{\ell}}\tilde{g}^{r\bar{s}}\partial_{\bar{q}}(T_{rk\bar{s}})F_{\bar{\ell}}F_p   +\tilde{g}^{p\bar{q}}\tilde{g}^{k\bar{\ell}}\tilde{g}^{r\bar{s}}\tilde{g}^{i\bar{j}}\partial_{\bar{q}} \tilde{g}_{i\bar{s}}T_{rk\bar{j}}F_{\bar{\ell}}F_p + \tfrac{1}{2}|\tilde{\nabla} \bar{\tilde{\nabla}} F|^2_{\tilde{\omega}} \\
& \ \ \ \ + \varepsilon |\partial F|^4_{\tilde{\omega}} - |\tilde{\Delta} F ||\partial F|^2_{\tilde{\omega}} \\
&\ge -C(\tr_\omega \tilde{\omega})^Bg^{i\bar{j}}\tilde{g}^{k\bar{\ell}}\tilde{g}^{p\bar{q}}\partial_i \tilde{g}_{k\bar{q}}\partial_{\bar{j}}\tilde{g}_{p\bar{\ell}} + \tfrac{1}{4}|\tilde{\nabla} \bar{\tilde{\nabla}} F|^2_{\tilde{\omega}} - C(\tr_\omega\tilde{\omega})^B|\partial F|^2_{\tilde{\omega}}  - C(\tr_\omega\tilde{\omega})^B .
\end{align*}
Now, we use the following computation in the proof of Equation (9.5) of \cite{tw15} for $\tilde{\Delta}\tr_\omega\tilde{\omega}$:
\begin{align*}
\tilde{\Delta}\tr_\omega \tilde{\omega}&=\tilde{g}^{p\bar{j}}\tilde{g}^{i\bar{q}}g^{k\bar{\ell}}\nabla_k\tilde{g}_{i\bar{j}}\nabla_{\bar{\ell}}\tilde{g}_{p\bar{q}}+2\text{Re}(\tilde{g}^{i\bar{j}}g^{k\bar{\ell}}T^p_{ki}\nabla_{\bar{\ell}}\tilde{g}_{p\bar{j}})+\tilde{g}^{i\bar{j}}g^{k\bar{\ell}}T^p_{ik}\overline{T^q_{j\ell}}\tilde{g}_{p\bar{q}}\\
& \ \ \ \ +g^{i\bar{j}}F_{i\bar{j}}-R +\tilde{g}^{i\bar{j}}\nabla_i \overline{T^\ell_{j\ell}}+\tilde{g}^{i\bar{j}}g^{k\bar{\ell}}\nabla_{\bar{\ell}}T^p_{ik}-\tilde{g}^{i\bar{j}}g^{k\bar{\ell}}\tilde{g}_{k\bar{q}}(\nabla_i\overline{T^q_{j\ell}}-R_{i\bar{\ell}p\bar{j}}g^{p\bar{q}})\\
& \ \ \ \ -\tilde{g}^{i\bar{j}}g^{k\bar{\ell}}T^p_{ik}\overline{T^q_{j\ell}}g_{p\bar{q}}.
\end{align*}
Converting the first term into covariant derivatives and applying Young's inequality, we have
\begin{align*}
\tilde{g}^{p\bar{j}}\tilde{g}^{i\bar{q}}g^{k\bar{\ell}}\nabla_k \tilde{g}_{i\bar{j}}\nabla_{\bar{\ell}}\tilde{g}_{p\bar{q}} \ge \tilde{g}^{p\bar{j}}\tilde{g}^{i\bar{q}}g^{k\bar{\ell}}\partial_k \tilde{g}_{i\bar{j}}\partial_{\bar{\ell}}\tilde{g}_{p\bar{q}}-\tfrac{\varepsilon}{2}\tilde{g}^{p\bar{j}}\tilde{g}^{i\bar{q}}g^{k\bar{\ell}}\partial_k \tilde{g}_{i\bar{j}}\partial_{\bar{\ell}}\tilde{g}_{p\bar{q}} - C(\tr_\omega \tilde{\omega})^{n} .
\end{align*}
Likewise, the second term can be bounded below by
\begin{align*}
2\text{Re}(\tilde{g}^{i\bar{j}}g^{k\bar{\ell}}T^p_{ki}\nabla_{\bar{\ell}}\tilde{g}_{p\bar{j}}) \ge -\tfrac{\varepsilon}{2}\tilde{g}^{p\bar{j}}\tilde{g}^{i\bar{q}}g^{k\bar{\ell}}\partial_k \tilde{g}_{i\bar{j}}\partial_{\bar{\ell}}\tilde{g}_{p\bar{q}} - C(\tr_\omega\tilde{\omega})^{n},
\end{align*}
and the fourth term by
\begin{align*}
g^{i\bar{j}}F_{i\bar{j}}&\ge - \tfrac{|\tilde{\nabla}\bar{\tilde{\nabla}}F|_{\tilde{\omega}}^2}{\delta} - C\delta(\tr_\omega \tilde{\omega})^2.
\end{align*}
It is straightforward to see that the remaining terms can be bounded below by $-C(\tr_\omega\tilde{\omega})^{n}$.
Choosing $B\ge n$ and $\delta = 4e^{-A(F)}N(B+1)(\tr_\omega\tilde{\omega})^{B}$, we arrive at the following:
\begin{align*}
\tilde{\Delta}\tr_\omega\tilde{\omega}&\ge (1-\varepsilon)\tilde{g}^{p\bar{j}}\tilde{g}^{i\bar{q}}g^{k\bar{\ell}}\partial_k\tilde{g}_{i\bar{j}}\partial_{\bar{\ell}}\tilde{g}_{p\bar{q}} -\tfrac{e^{A(F)}}{4N(B+1)(\tr_\omega\tilde{\omega})^B}|\tilde{\nabla}\bar{\tilde{\nabla}}F|^2_{\tilde{\omega}} -C(\tr_\omega\tilde{\omega})^{B+2}.
\end{align*}
Observe that
\begin{align*}
\tilde{\Delta} (\tr_\omega \tilde{\omega})^{B+1} &= (B+1)B(\tr_\omega \tilde{\omega})^{B-1}|\partial \tr_\omega \tilde{\omega}|^2_{\tilde{\omega}} + (B+1)(\tr_\omega \tilde{\omega})^{B}\tilde{\Delta} \tr_\omega \tilde{\omega}\\
&\ge (B+1)(\tr_\omega \tilde{\omega})^{B}\tilde{\Delta} \tr_\omega \tilde{\omega} .
\end{align*}
Choosing $N$ sufficiently large and letting $Q:= e^{A(F)}|\partial F|^2_{\tilde{\omega}} + N(\tr_\omega \tilde{\omega})^{B+1}$, we have
\begin{align*}
\tilde{\Delta} Q &= \tilde{\Delta}(e^{A(F)}|\partial F|^2_{\tilde{\omega}} + N(\tr_\omega \tilde{\omega})^{B+1})\\
& \ge -C(\tr_\omega \tilde{\omega})^B |\partial F|^2_{\tilde{\omega}} - C(\tr_\omega \tilde{\omega})^{2B+2}\\
& \ge -C(\tr_\omega \tilde{\omega})^{B+1}(|\partial F|^2_{\tilde{\omega}} +N(\tr_\omega \tilde{\omega})^{B+1})\\
& \ge -C(\tr_\omega \tilde{\omega})^{B+1} Q  .
\end{align*}
The rest of the proof leading to the $L^\infty$ bound on $Q$ follows using Moser iteration and several instances of the H\"older inquality and the Sobolev inequality with respect to the reference metric $\omega$, see Section 4 of \cite{cc17}. The constants and powers of the trace differ slightly from the K\"ahler case, but do not affect the iteration method. In addition, showing an $L^1$ bound on the quantity $Q$ is straightforward since the bound for $|\partial F|^2_{\tilde{\omega}}$ holds the same way as in (4.35) of \cite{cc17} and the $L^1$ bound on $(\tr_\omega \tilde{\omega})^{B+1}$ follows using the $L^{B+1}$ norm we obtained in Section 5.
\end{proof}
Combining this upper bound on $\tr_\omega\tilde{\omega}$ with the fact that we already have a lower bound establishes the quasi-isometry of $\omega$ and $\tilde{\omega}$. The higher order estimates of $\varphi$ can then be obtained using a bootstrapping argument as in the proof of Proposition 1.2 in \cite{cc17} where the $C^{2,\alpha}$ estimate is obtained using the result in \cite{tssy}.

\section*{Acknowledgements}
The author would like to thank her thesis advisor Ben Weinkove for suggesting this problem and for his continued support, encouragement and countless helpful insights. The author is also grateful to Viktor Burghardt, Gregory Edwards, Nicholas McCleerey and Cristiano Spotti for several useful conversations.

\bibliography{bib_for_all}
\end{document}